\documentclass[12pt]{article}
\usepackage{amsmath}
\usepackage{amssymb}
\usepackage{amsthm}
\usepackage{graphicx} 
\usepackage[legalpaper, margin=1in]{geometry}
\usepackage{enumitem}

\newtheorem{theorem}{Theorem} 
\newtheorem{corollary}[theorem]{Corollary}
\newtheorem{lemma}[theorem]{Lemma}
\newtheorem{proposition}[theorem]{Proposition}

\theoremstyle{definition}
\newtheorem{definition}[theorem]{Definition}
\newtheorem{question}[theorem]{Question}
\newtheorem{remark}[theorem]{Remark}

\theoremstyle{remark}

\newcommand{\cantor}{\mathbf C}

\title{The converse to Borsuk’s result on fans fails} 
\author{
Benjamin Vejnar\footnote{https://orcid.org/0000-0002-2833-5385} \footnote{
Supported by the grant GACR 24-10705S
}\\
Department of Mathematical Analysis\\ 
Faculty of Mathematics and Physics, Charles University\\
Prague, Czechia\\
E-mail: vejnar@karlin.mff.cuni.cz}

\begin{document}
\maketitle

\renewcommand{\thefootnote}{}



\begin{abstract}
A fan is an arc-wise connected hereditarily unicoherent continuum with exactly one branching point.
By a result of Borsuk, every fan is a 1-dimensional continuum that can be expressed as the union of a family of arcs, each pair of which intersects in the branching point.
In this paper, we prove that the converse does not hold by providing a more general result.
\end{abstract}

\section{Introduction}
A dendroid is an arc-wise connected hereditarily unicoherent continuum, and a fan is a dendroid with exactly one branching point.
In 1954, while proving the fixed point property for dendroids, Borsuk proved that every dendroid is hereditarily decomposable \cite{Borsuk}. Consequently, the topological dimension of any dendroid is one (see e.g. \cite{CharatonikDednroids}). In the same paper, Borsuk proved that the closure of any injective continuous image of a ray in any dendroid is an arc. It follows that every fan can be expressed as the union of arcs, each pair of which intersects in a distinguished point.
In 1963, Borsuk defined a \textbf{broom} to be a compact 1-dimensional space $X$ that can be expressed as the union of a family $\mathcal L$ consisting of arcs, such that any two distinct arcs $K, L\in\mathcal L$ intersect only in a point $t\in X$ \cite{BorsukBrooms}.
While proving that there is a countable broom (i.e., $\mathcal L$ is countable) which cannot be embedded into the plane, he noted that a countable broom is always a fan.
Nothing is mentioned about the case when $\mathcal L$ is uncountable. 
It is natural to ask:

\begin{question}\label{questionmain}
Is every broom a fan?
\end{question}

This question was explicitly stated in \cite[Problem 3.4]{phans}, where the authors proved that a broom with some additional properties is a fan. The question is also rephrased in \cite{Lipham}.
The aim of this paper is to provide a negative answer to Question \ref{questionmain} by proving a more general result:

\begin{theorem}\label{mainthm}
Let $C$ be the Cantor set and $T$ be a continuum for which there is a continuous surjective map $f:C\to T$ where $|f^{-1}(t)|\leq2$, $t\in T$, and $|f^{-1}(t)|=2$ only for countably many points $t\in T$.

Then, there exists a 1-dimensional continuum $X$ with a point $t\in X$ 
and a family $\mathcal L$ consisting of arcs in $X$ such that 
\begin{itemize}[noitemsep]
 \item $K\cap L=\{t\}$ for every pair of distinct arcs $K,L\in\mathcal L$,
 \item $t$ is an end-point of every $L\in\mathcal L$,
 \item $X$ contains a homeomorphic copy of $T$.
 \item $X$ contains a homeomorphic copy of the circle,
 \item $|\mathcal L|\geq 3$.
\end{itemize}
\end{theorem}

Theorem \ref{mainthm} can be applied, e.g., to $T=[0,1]$ and $f:C\to T$ the standard Cantor map. We get a broom $X$ that contains a simple closed curve (in fact, many of them); hence, $X$ cannot be a fan. Thus, we get the following corollary.

\begin{corollary}
There exists a broom that is not a fan.
\end{corollary}

Using Remark \ref{remarkmore}, we can apply Theorem \ref{mainthm} to the Sierpiński triangle $T$ to obtain a broom that contains a homeomorphic copy of $T$, thus being very different from a fan.
Our result also answers \cite[Problem 4.17]{phans} in the negative. See Theorem \ref{theoremanotherproblem} for the details.

\section{Preliminaries}

We denote $\omega=\{0, 1,\dots\}$ and $\mathbb N=\{1, 2,\dots\}$.
We use indices $s\in \omega^{<\omega}$, i.e. $s=(s_0,\dots, s_{n-1})$ is a finite sequence of non-negative integers. We denote by $|s|$ the length of $s$. Note that $s$ can also be an empty sequence. Also, if $|s|=1$ we identify $s$ with $s_0$.

A continuum is a nonempty, compact, connected, metrizable topological space. A simple closed curve is any space homeomorphic to the unit circle, and an arc is any space homeomorphic with $[0,1]$. A continuum is hereditarily unicoherent, provided that the intersection of two subcontinua is either empty or a continuum.
A dendroid is an arc-wise connected, hereditarily unicoherent continuum. A simple triod with top point $t$ is a continuum that can be expressed as a union of three arcs, each pair of which intersects exactly in $\{t\}$ and $t$ is an end-point of each of these arcs.  We use \cite{Nadler} as a standard reference for notions in Continuum theory.

A topological space is 0-dimensional if it has a base formed by sets that are open and closed at the same time. A space is 1-dimensional if it is not 0-dimensional and has an open base formed by sets with a 0-dimensional boundary. See \cite{EngelkingDim} for more on Dimension theory in topology.

Every dendroid is a 1-dimensional continuum. A point in a dendroid is called a branching point if it is the top of a simple triod contained in it (see \cite{CharatonikDednroids}). A fan is a dendroid with exactly one branching point.
For an exposition on fans, see the paper by J. J. Charatonik \cite{CharatonikOnFans}.
We use the terminology of Borsuk for the notion of a broom:

\begin{definition}
A 1-dimensional continuum $X$ is called a broom if there is a point $t\in X$ and a family $\mathcal L$ consisting of arcs in $X$ such that 
\begin{itemize}[noitemsep]
\item $|\mathcal L|\geq 3$,
\item $X=\bigcup\mathcal L$,
\item $K\cap L=\{t\}$, for every pair of distinct arcs $K, L\in \mathcal L$.
\item the point $t$ is an end-point of every $L\in\mathcal L$.
\end{itemize}
\end{definition}
It can be easily seen that if we drop the last condition, we get an equivalent notion.
The reason to include it is that in the last section we state a question about the uniqueness of the family $\mathcal L$ in a broom.
Let us note that the name “phan" is used in \cite{phans} instead of a broom.

The following result is usually referred to as the Countable sum theorem for topological dimension \cite[Theorem 1.5.3]{EngelkingDim}. We will be using it only in the case $n=1$.


\begin{lemma}\label{countablesumtheorem}
The countable union of at most $n$-dimensional closed sets is at most $n$-dimensional.
\end{lemma}

\begin{figure}
\centering
\includegraphics[width=\linewidth]{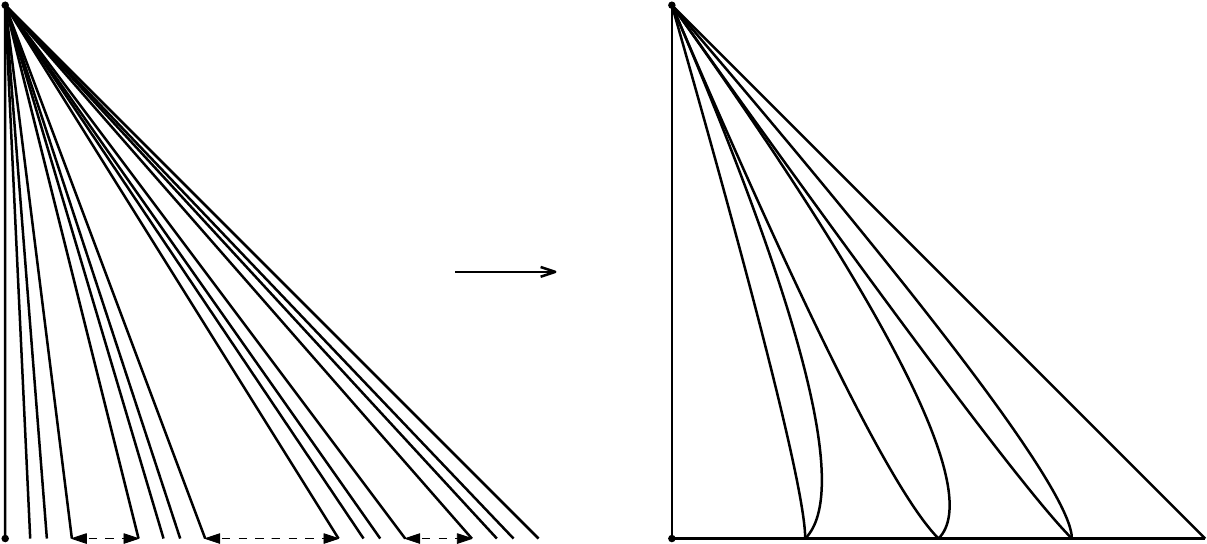} %
\caption{\label{FigY} The quotient map $q: F\to Y$.}
\end{figure}

\section{Proof of Theorem \ref{mainthm}}\label{mainsection}

For the readers convenience, we describe the proof of Theorem \ref{mainthm} in the case $T=[0,1]$ and using a special choice of $f$. The proof of the full result is presented at the end of this section.

\subsubsection*{Notation}

Let $I=[0,1]$ and $C\subseteq I$ be the standard Cantor set. Denote by $f:C \to I$ the Cantor map, i.e. the map for which $f(\sum 2a_i /3^i)=\sum a_i/2^i$, where $a_i\in\{0,1\}$. This is at most a 2-to-1 map. Denote by $D$ the collection of all dyadic numbers in $(0, 1)$.  
The set $D$ is countable, so we can enumerate it as $\{d_n: n\in\mathbb N\}$.
Then $f^{-1}(d_n)$ is a two point set, and we fix $a_n, b_n\in C$ with $f^{-1}(d_n)=\{a_n, b_n\}$, $n\in\mathbb N$.
On the other hand, $f^{-1}(d)$ is a singleton for $d\in I\setminus D$.

\subsubsection*{The Cantor fan $F$}

Let $F$ be the cone over $C$, i.e., $F$ is the quotient of $C\times I$ where $C\times \{1\}$ is pushed to a point, which will be called $t$. The continuum $F$ is usually called a Cantor fan with top point $t$. 
For $c\in C$ and $a<b\in I$, we denote by $I_c[a,b]$ the set $\{c\}\times [a,b]$; analogously, we use $I_c[a,b)$ for intervals open from the right side, i.e., $\{c\}\times [a,b)$, etc.
We will also shorten $I_c=I_c[0,1]$ and identify these sets with their images under the quotient map $C\times I \to F$.

\subsubsection*{The continuum $Y$ (see Figure \ref{FigY})}

Now, consider the quotient $Y$ of $F$, where two points $(a,0), (b,0)\in F$ are identified if and only if $f(a)=f(b)$. Let $q: F\to Y$ be the quotient map. 
Also denote $B=q(C\times \{0\})$, which is an arc in $Y$. 

\subsubsection*{The continuum $Z$ as a quotient of $Y$ (see Figure \ref{FigZ})}

Define $Z$ as a quotient of $Y$ in the following way.
Identify the arc $B$ with the arc $q(I_0)$ by the homeomorphism that maps $q(c,0)$ to $q(0,f(c))$, $c\in C$. Let $p: Y\to Z$ be the quotient map. Denote $A=p(B)$.

\subsubsection*{The compact space $W$}

We define an indexing set 
\[\Lambda=\{s\in\omega^{<\omega}: s\neq\emptyset, s_0\neq 0\}=\mathbb N\times \omega^{<\omega}.\]
The form of this indexing set reflects the fact 
that once we produce our final space $X$ as a quotient of $W$, we will need $\mathbb N$ many gluings in the first steps, and in the additional steps, we will always need one more space corresponding to the index $s$ whose last entry is zero.

Let $Z_s=Z$ for $s\in \Lambda$.
Let $W$ be the disjoint union of $Y$ and all the $Z_s$, topologized in such a way that $Z_s$, $s\in \Lambda$, forms a null sequence converging to the point $t$.
In other words, the disjoint sum $\bigcup \{Z_s: s\in \Lambda\}$ (which is a non-compact, locally compact, separable space) is compactified by the one point $t\in Y$. 
Thus, $W$ is clearly a compact space.
As we want $Z_s$, $s\in\Lambda$, to be disjoint copies of $Z$ in $W$, we adopt the notation of $p_s: Y\to Z_s$, $t_s=p_s(q(t))$, $A_s=p_s(B)$ for the corresponding objects in $Z_s$, in order to distinguish them.

\begin{figure}
\centering
\includegraphics[width=\linewidth]{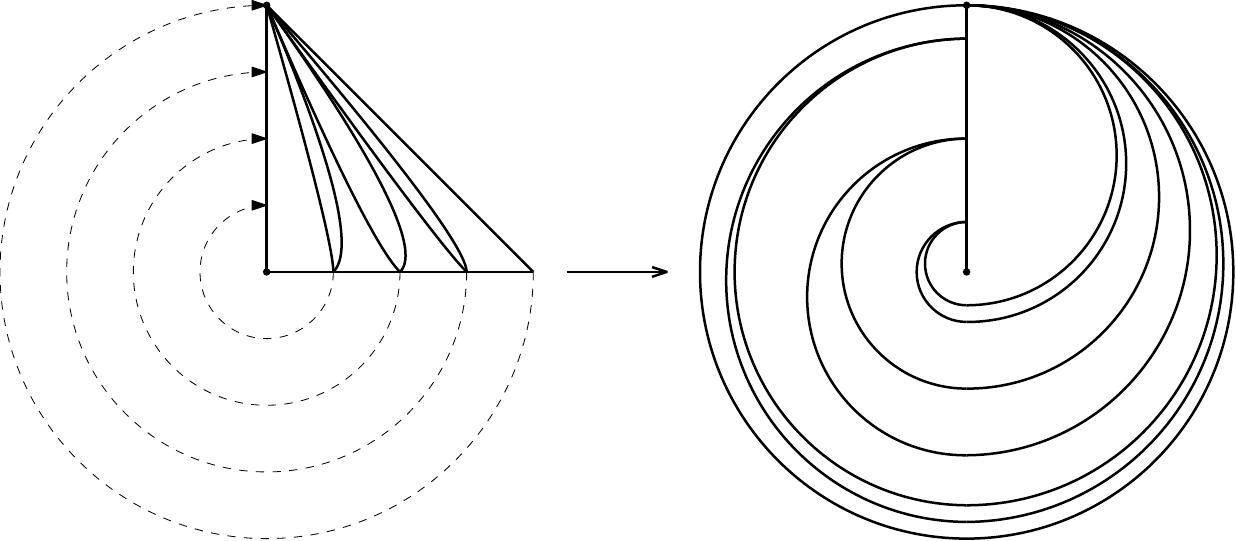} %
\caption{\label{FigZ} The quotient map $q: Y\to Z$.}
\end{figure}

\subsubsection*{The space $X$ as a quotient of $W$ (see Figure \ref{FigW})} 
Fix real numbers $r_n\in (0,1)$, $n\in\mathbb N,$ converging to 1.
First, identify in $W$ every arc $q(I_{b_n}[r_n,1])$ with the arc $A_n\subseteq Z_n$ using a homeomorphism, say $h_{n}$, that maps $q(t)$ to $t_n$, $n\in\mathbb N$.

Consequently, for every $s\in \Lambda$ and $n\in\mathbb N$,
identify in $W$ the arc
$p_sq(I_{b_n}[r_n, 1])$ with the arc $A_{s\smallfrown n}\subseteq Z_{s\smallfrown n}$ by a homeomorphism, say $h_{s\smallfrown n}$, of these arcs that maps $t_s$ to $t_{s\smallfrown n}$.

Additionally, identify the arc $p_sq(I_1[1/2, 1])$ with the arc $A_{s\smallfrown 0}\subseteq Z_{s\smallfrown 0}$ by a homeomorphism, say $h_{s\smallfrown 0}$, of these arcs that maps $t_s$ to $t_{s\smallfrown 0}$.

Finally, let $X$ be the quotient of $W$ and denote by $u:W\to X$ the quotient map.

\subsubsection*{The family of arcs $\mathcal L$}

We describe the family of arcs $\mathcal L=\bigcup \{\mathcal L_s: s \in \Lambda \cup\{\emptyset\}\}$.
We will proceed in different ways for $s=\emptyset$ and $s\in\Lambda$.

\subsubsection*{The family $\mathcal L_\emptyset$ of arcs in $Y$}

Let $\mathcal L_\emptyset$ consist of all the arcs in $Y$ of the form $q(I_c)$, whenever $c\in C\setminus \{a_n, b_n: n\in \mathbb N\}$ (i.e. $f^{-1}(f(c))$ is a singleton.) 
We also add to $\mathcal L_\emptyset$ all the arcs 
$q(I_{a_n} \cup I_{b_n}[0,r_n])$, $n\in\mathbb N$.


\subsubsection*{The family $\mathcal L_s$ of arcs in $Z_s$}

Fix $s\in\Lambda$, $|s|=n$, $n\in\mathbb N$.
Let $\mathcal L_s$ consist of all the arcs in $Z_s$ of the form $p_sq(I_c)$, whenever $c\in C\setminus \{0,1, a_n, b_n: n\in\mathbb N\}$ (i.e. $f^{-1}f(c)$ is a singleton and $c\neq 0$ and $c\neq 1$). Add to this family $\mathcal L_s$ all the arcs of the form $p_sq(I_{a_n} \cup I_{b_n}[0, r_n])$.
Finally, add one more arc $p_sq(I_1[0, 1/2])$.
Note that $p_sq(I_1[0, 1])$ is a simple closed curve, not an arc! This is why this exception with $c=1$ is crucial.

\begin{figure}
\centering
\includegraphics[width=\linewidth]{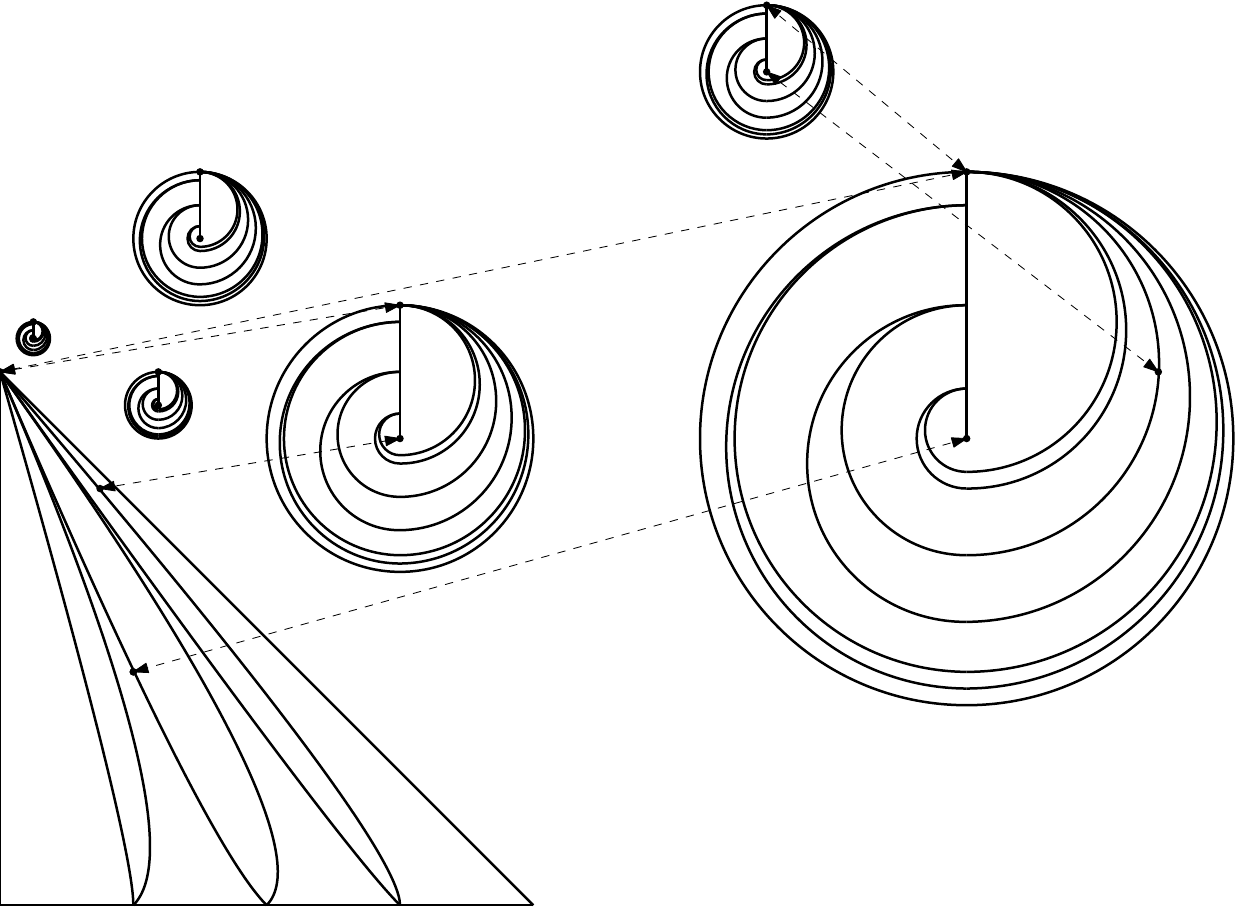} %
\caption{\label{FigW} The domain of the quotient map $u: W\to X$.}
\end{figure}

\subsubsection*{Verification that everything works}

The restrictions $u|Y$ as well as $u|Z_s$, $s\in\Lambda$, are injective. Hence, we can consider $Y$ and $Z_s$ as subspaces of $W$.
The same applies to every $L\in\mathcal L$. 
We just note that $Y\cap Z_n=A_n$ and $Z_s\cap Z_{s\smallfrown n}=A_{s^\smallfrown n}$.

The space $X$ is a quotient of a compact metric space $W$ by an upper semi-continuous decomposition, since $A_n$ converges to $\{t\}$. Thus, $X$ is a compact metrizable space.
Since moreover $t\in Y\cap\bigcap_{s\in\Lambda} Z_s$, the space $X$ is a continuum.

\subsubsection*{The space $X$ is covered by $\mathcal L$}
Let $x\in X$. We can suppose that $x\neq t$. Distinguish two cases and subcases. First, if $x\in Y$, then either a) $x\in\bigcup\mathcal L_\emptyset$ or b) $x\in A_n$, for some $n\in\mathbb N$, but $x$ is not an end-point of $A_n$. In the b) case, $x$ is contained in an arc from $\mathcal L_n$.
Second, if $x\in Z_s$ for some $s\in\Lambda$, then either a) $x\in\bigcup\mathcal L_s$ or b) $x\in A_{s\smallfrown n}$ for some $n\in\omega$, and $x$ is not an end-point of $A_{s\smallfrown n}$.
In the b) case, we also get that $x$ is contained in an arc from  $\mathcal L_{s^\smallfrown n}$.

\subsubsection*{Distinct arcs from $\mathcal L$ intersect in $t$ only}
We need to verify first that $K\cap L=\{t\}$ for distinct $K, L\in\mathcal L$.
Clearly, if we have distinct $K, L\in \mathcal L_s$, $s\in\Lambda$, then $K\cap L=\{t\}$. This simply holds also in the case of $s=\emptyset$.

Also $u(Z_s)\cap u(Z_{s^\prime})=\{t\}$ unless $s$ is a prolongation of $s^\prime$ by one element or vise versa.
Hence we may suppose that $K\in \mathcal L_s$ and $L\in \mathcal L_{s\smallfrown n}$ with $s\in\Lambda$ and $n\in\omega$. Since $u(Z_s)\cap u(Z_{s\smallfrown n})=u(B_{s\smallfrown n})$.
However, the only $K\in \mathcal L_s$ that intersects the arc $u(B_{s\smallfrown n})$ in a point other than $t$, intersects it in the opposite end-point. That end-point is, however, not used by any $L\in\mathcal L_{s\smallfrown n}$.

Similar argumentation also works in the last case when $K\in\mathcal L_\emptyset$ and $L\in\mathcal L_s$, $s\in\Lambda$.




\subsubsection*{The continuum $X$ has dimension 1}

The Cantor fan $F$ is well known to be 1-dimensional.
The continuum $Y$ can be expressed as a countable union of homeomorphic copies of the Cantor fan and the arc $B$. 
Hence, by Lemma \ref{countablesumtheorem}, $Y$ is 1-dimensional.
Similarly, the continuum $Z$ is a countable union of homeomorphic copies of $F$, and thus it is 1-dimensional as well. Finally, $X$ is a countable union of $u(Y)$ and $u(Z_s)$, $s\in\Lambda$, which are homeomorphic copies of $Y$ or $Z$. Hence, again by Lemma \ref{countablesumtheorem}, we get that $X$ is 1-dimensional.

\subsubsection*{The continuum $X$ contains simple closed curves}

The union $q(I_{a_n})\cup q(I_{b_n})$ is a simple closed curve for every $n\in\mathbb N$.

\subsubsection*{Generalizing the whole proof for arbitrary $T$ and $f$}

We can modify the whole construction as follows.
Keep the space $Z$ as it is, but modify the construction of $Y$ naturally by using a continuum $T$ for which there is a continuous surjective map $f:C\to T$ where $f^{-1}(t)$ is nondegenerate only for countably many points $t\in T$, and in these cases $|f^{-1}(t)|=2$.
Note that the dimension of $T$ is 1 by \cite[Problem 1.7.F c)]{EngelkingDim}.
Then we can analogously define the space $Y=Y_f$ as a quotient of the Cantor fan $F$ by the identification of points $(a,0), (b,0)\in F$ whenever $f(a)=f(b)$. Proceed in the construction in the same vein by defining $W$ and $X$ using $Y_f$ and $Z$.
In this way, the construction yields a broom that contains $T$.

\section{Concluding remarks and questions}

\begin{remark}\label{remarkmore}
By \cite{WinklerTriangle}, the Sierpiński triangle $T$ is a continuous image of the Cantor set under a continuous surjective map $f$, such that the preimages of points are of cardinality at most two, and only countably many points have a preimage of cardinality exactly two.
Thus, Theorem \ref{mainthm} can be applied to the Sierpiński triangle $T$ and a map $f$. See also \cite{VejnarTriangle} for other spaces of this kind. 
\end{remark}

\begin{remark}
By the classical Moore plane decomposition theorem, every upper semi-continuous decomposition of the plane, whose classes are continua that do not separate the plane, is homeomorphic to the plane.
It follows from the Moore theorem that both spaces $Y$ and $Z$ from Section \ref{mainsection} can be realized in the plane (for the case $T=[0,1]$). The argument uses a decomposition whose nontrivial classes are arcs (the dashed arcs in Figures \ref{FigY} and \ref{FigZ}), which clearly do not separate the plane.
Consequently, $W$ is embeddable into the plane as well.
However, the gluing towards the space $X$ suggests that the continuum $X$ cannot be embedded into the plane.
\end{remark}

\begin{question}
Is every planar broom a fan?
\end{question}

\begin{remark}
It is well known that every family of disjoint simple triods in the plane is countable.
Disjoint families of simple triods in every broom $X$ obtained by Theorem \ref{mainthm} are countable because every triod in $X$ that does not contain the point $t$ must intersect an arc $A_s$, $s\in\Lambda$, in a subarc.
However, we can easily construct a broom with an uncountable family of disjoint simple triods. Just start with any broom $X$ with the top point $t$ constructed in Theorem \ref{mainthm}. We know that there is a simple triod $S\subseteq X\setminus \{t\}$. Let $C$ denote the Cantor set. Then consider the quotient of $C\times X$ where $C\times\{t\}$ is pushed to a point. This quotient space is clearly a broom, and $\{\{c\}\times S: c\in C\}$ is an uncountable family of disjoint triods.
\end{remark}

\begin{remark}\label{differentArcs}
Let $X$ be any broom from Theorem \ref{mainthm} and 
let $\mathcal L, \mathcal L^\prime$ be two families of arcs witnessing that $X$ is a broom. Then $\mathcal L =\mathcal L^\prime$. The same is true for every fan $X$.
\end{remark}


\begin{question}
Does there exist a broom $X$ and distinct families $\mathcal L, \mathcal L^\prime$ of arcs in $X$, both wittnesing that $X$ is a broom?
\end{question}

\begin{remark}
A space is said to have the fixed point property, if every continuous selfmap has a fixed point. Note that the fixed point property is preserved by retracts.
By \cite[Theorem 1, p.354]{Kuratowski} (see also \cite[Proposition 2.4]{IllanesKrupski}),
every simple closed curve contained in a 1-dimensional continuum $X$ is a retract of $X$.
Since all the brooms constructed in Theorem \ref{mainthm} contain simple closed curves, none of them has the fixed point property. This is in contrast to the fact that dendroids have the fixed point property.
\end{remark}

\begin{question}
Does there exist a broom with the fixed point property, that is not a fan?
\end{question}

In \cite{phans}, a continuum $X$ is called a \textbf{Carolyn pan} if there is a point $t\in X$ and a family $\mathcal L$ consisting of arcs in $X$, such that 
\begin{itemize}[noitemsep]
    \item $K\cap L=\{t\}$ for distinct arcs $K, L\in\mathcal L$,
    \item $X=\bigcup \mathcal L$,
    \item for every continuum $C\subseteq X\setminus \{t\}$ there is $L\in\mathcal L$, such that $C\subseteq L$.
\end{itemize}

We give a negative answer to \cite[Problem 4.17]{phans} by the following theorem.

\begin{theorem} \label{theoremanotherproblem}
The space $X$ obtained by Theorem \ref{mainthm} is not a Carolyn pan.
\end{theorem}

\begin{proof}
Note that $X$ contains simple closed curves, so it is not a fan.
Suppose for contradiction that $X$ is a Carolyn pan.
Note that $X$ is 1-dimensional.
Since every 1-dimensional Carolyn pan is a fan by \cite[Theorem 4.15]{phans}, we get a contradiction.

An alternative self-contained proof of the theorem is as follows.
Suppose for contradiction that $X$ is a Carolyn pan. Hence, there is $t\in X$ and a family $\mathcal L$ as in the definition. 
Notice that $X$ contains a simple closed curve $S\subseteq X$ with $t\in S$. 
Then $S\setminus \{t\}$ contains arcs $C_n$ whose both end-points approach $t$ as $n\to\infty$, $t\notin C_n$, and $C_1\subseteq C_2\subseteq \dots$.
Choose $L\in\mathcal L$ with $C_1\subseteq L$. Clearly, for sufficiently large $n\in\mathbb N$, $C_n$ is not a subset of $L$. This is a contradiction with $X$ being a Carolyn pan.
\end{proof}

\section*{Acknowledgement}
I am grateful to Iztok Bani\v{c}, who encouraged me to solve the main problem of this paper during the 3rd Maribor Conference on
Topological Dynamics and Continuum Theory.
Also, special thanks belong to the authors of \cite{phans} who verified the early proof of the main result of this paper during the 59th Spring Topology and Dynamics Conference in Birmingham.
The MathOverflow discussion of a 1-dimensional continuum that is the disjoint union of arcs \cite{Mathoverflow} was also helpful to me.
\bibliographystyle{abbrv}
\bibliography{sample}

\end{document}